\documentclass[11pt,a4paper]{amsart}
\newcommand*\patchAmsMathEnvironmentForLineno[1]{%
  \expandafter\let\csname old#1\expandafter\endcsname\csname #1\endcsname
  \expandafter\let\csname oldend#1\expandafter\endcsname\csname end#1\endcsname
  \renewenvironment{#1}%
     {\linenomath\csname old#1\endcsname}%
     {\csname oldend#1\endcsname\endlinenomath}}%
\newcommand*\patchBothAmsMathEnvironmentsForLineno[1]{%
  \patchAmsMathEnvironmentForLineno{#1}%
  \patchAmsMathEnvironmentForLineno{#1*}}%
\AtBeginDocument{%
\patchBothAmsMathEnvironmentsForLineno{equation}%
\patchBothAmsMathEnvironmentsForLineno{align}%
\patchBothAmsMathEnvironmentsForLineno{flalign}%
\patchBothAmsMathEnvironmentsForLineno{alignat}%
\patchBothAmsMathEnvironmentsForLineno{gather}%
\patchBothAmsMathEnvironmentsForLineno{multline}%
}

\usepackage{amsthm, amsfonts, amssymb, amsmath, latexsym, enumerate, array}
\usepackage[all]{xy}

\ifx\pdftexversion\undefined
\usepackage[dvips]{graphicx}
\DeclareGraphicsExtensions{.jpg,.eps,.pnm}
\else
\usepackage[pdftex]{graphicx}
\fi

\usepackage[latin1]{inputenc}
\usepackage{hyperref,cite,mdwlist,mathrsfs}
\usepackage[pagewise]{lineno}

\usepackage[left=3.2cm,top=2.7cm,right=3.2cm]{geometry}

\newcommand{\mathbold}{\mathbf}

\newtheorem{thm}{Theorem}[section] \newtheorem{lem}[thm]{Lemma}
\newtheorem{corol}[thm]{Corollary} \newtheorem{prop}[thm]{Proposition}
\newtheorem{claim}[thm]{Claim}
\newtheorem{main}{Theorem}
\newtheorem*{thm*}{Theorem} \newtheorem*{cnj*}{Conjecture}

\theoremstyle{definition} \newtheorem{rmk}[thm]{Remark}
\newtheorem{eg}[thm]{Example} \newtheorem{dfn}[thm]{Definition}

\newtheorem*{conj*}{Conjecture}

\newcommand{\sO}{\mathscr{O}}
\newcommand{\sT}{\mathscr{T}}
\newcommand{\cExt}{\mathcal{E}xt}
\newcommand{\sE}{\mathscr{E}}

\newcommand{\sF}{\mathscr{F}}
\newcommand{\sK}{\mathscr{K}}
\newcommand{\sC}{\mathscr{C}}
\newcommand{\sL}{\mathscr{L}}

\newcommand{\sJ}{\mathscr{J}}

\newcommand{\W}{\mathrm{W}}

\newcommand{\cT}{\mathcal{T}}

\newcommand{\cS}{\mathcal{S}}

\newcommand{\cI}{\mathcal{I}}

\newcommand{\cHom}{\mathcal{H}om}

\DeclareMathOperator{\Ext}{Ext}
\DeclareMathOperator{\Hom}{Hom} 
\DeclareMathOperator{\im}{Im} \DeclareMathOperator{\cok}{Cok}

\DeclareMathOperator{\HH}{H}

\newcommand{\tra}{{}^{\mathrm{t}}}

\newcommand{\Z}{\mathbb Z} 
\newcommand{\F}{\mathbb F} \newcommand{\LL}{\mathbb L}
 
 \newcommand{\p}{\mathbb P}

\newcommand{\RRHHom}{\mathbold{R}\mathcal{H}om}
\newcommand{\RR}{\mathbold{R}}
\newcommand{\Db}{\mathbold{D}^b}

\DeclareMathOperator{\ts}{\otimes}
\newcommand{\mono}{\hookrightarrow}
\newcommand{\epi}{\twoheadrightarrow}
\newcommand{\xr}{\xrightarrow}
\newcommand{\kk}{\mathbf{k}}

\newcommand{\shift}{\xr{{}_{[1]}}}

\SelectTips{cm}{} \newdir{ >}{{}*!/-5pt/@{>}}
\numberwithin{equation}{section}

\allowdisplaybreaks[1]
\begin{document}
\linenumberdisplaymath


\title{Hyperplane arrangements of Torelli type}

\author{Daniele Faenzi}
\email{\tt daniele.faenzi@univ-pau.fr}
\address{Universit\'e de Pau et des Pays de l'Adour \\
  Avenue de l'Universit\'e - BP 576 - 64012 PAU Cedex - France}
\urladdr{{\url{http://univ-pau.fr/~faenzi/}}}

\author{Daniel Matei}
\email{{\tt Daniel.Matei@imar.ro}}
\address{Institute of Mathematics "Simion Stoilow" of the Romanian Academy
I.M.A.R., Bucharest, Romania \\  P.O. Box 1-764, RO-014700, Bucharest, Romania}
\urladdr{{\url{http://www.imar.ro/~dmatei/}}}

\author{Jean Vallès}
\email{{\tt jean.valles@univ-pau.fr}}
\address{Universit\'e de Pau et des Pays de l'Adour \\
  Avenue de l'Universit\'e - BP 576 - 64012 PAU Cedex - France}
\urladdr{{http://web.univ-pau.fr/~jvalles/jean.html}}



\thanks{All authors partially supported by ANR-09-JCJC-0097-0
  INTERLOW.
  D.M. has been partially supported
  by grant CNCSIS PNII-IDEI 1189/2008.}

\begin{abstract}
We give a necessary and sufficient condition in order for a hyperplane
arrangement to be of Torelli type, namely that it is recovered as 
the set of unstable hyperplanes of its Dolgachev sheaf of
logarithmic differentials.
Decompositions and semistability of non-Torelli arrangements are investigated.
\end{abstract}

\maketitle

\section*{Introduction}

An arrangement of hyperplanes in $\p^n$ is the union $D$ of $\ell$
distinct hyperplanes $H_1,\ldots,H_\ell$ of $\p^n$, so
$H_i=\{f_i=0\}$, where $f_i$ is a linear form.
The topology, the geometry, and the combinatorial properties of the
pair $(\p^n,D)$ are interesting from many points of
view, we refer to \cite{orlik-terao:arrangements-hyperplanes} for a
comprehensive treatment. Let us only mention that Arnold,
in his foundational paper \cite{arnold:colored}, first used the
algebra of differential forms $d f_i/f_i$, to
give an explicit description of the cohomology ring of $\p^n \setminus
D$, an approach generalized by Brieskorn, see \cite{brieskorn:tresses}.

More generally, Deligne defined and extensively used in
\cite{deligne:equations} the sheaf $\Omega_X(\log D)$ of forms with logarithmic
poles along $D$, when $D$ is a normal crossing divisor of a smooth
variety $X$, while Saito in
\cite{saito:logarithmic} gave a definition of $\Omega_X(\log D)$ for
more general divisors. Anyway $\Omega_X(\log D)$ is the dual of the
sheafified derivation module, and as such it is a reflexive sheaf, in fact locally free
if $D$ is normal crossing.

Let again $D$ be a hyperplane arrangement with normal crossings (also
called a {\it generic} arrangement, namely $D$ is such that any $k$ hyperplanes meet along a $\p^{n-k}$).
The sheaf $\Omega_{\p^n}(\log(D))$ is then associated to $D$. The main
question asked (and solved) by Dolgachev and Kapranov in
\cite{dolgachev-kapranov:arrangements}, is whether 
one can reconstruct $D$ from
$\Omega_{\p^n}(\log(D))$. We say that $D$ is a {\it Torelli arrangement} in
this case (or simply $D$ is Torelli). They proved that if $\deg(D)\ge 2n+3$, then $D$ is Torelli
if and only if $D$ do not osculate a rational normal curve.
The result was extended to the range $\mathrm{deg}(D)\ge n+2$ in \cite{valles:instables}.

However this result only covers generic arrangement, while the
most interesting arrangements are far from being so.
On the other hand, Catanese-Hosten-Khetan-Sturmfels in \cite{catanese-hochsten-khetan-sturmfels}
and Dolgachev in \cite{dolgachev:logarithmic} defined a subsheaf $\tilde{\Omega}_{\p^n}(\log(D))$ of
$\Omega_{\p^n}(\log(D))$,
fitting in the residue exact sequence:
\[
 0 \to \Omega_{\p^n} \to \tilde{\Omega}_{\p^n}(\log(D)) \to
 \bigoplus_{i=1, \ldots, \ell} \sO_{H_i} \to 0. 
\]
Dolgachev in \cite{dolgachev:logarithmic} formulated the Torelli
problem for the sheaf $\tilde{\Omega}_{\p^n}(\log(D))$, and proposed
the following conjecture:
\begin{conj*}[Dolgachev]
   Assume $\tilde{\Omega}_{\p^n}(\log(D))$ is a semi-stable sheaf in
   the sense of Gieseker. Then $D$ is Torelli if and only if the
   points given by the $H_i$'s in the dual $\p^n$ do not belong to a
   {stable rational curve} of degree $n$.
\end{conj*}
A {\it stable rational curve} here means a connected curve of
arithmetic genus $0$ which is the union of $s$ smooth
rational curves $C_1,\ldots,C_s$, with $\deg(C_i)=d_i$ and $d_1+\cdots+d_s=n$, each $C_i$
spanning a $\p^{d_i}$, and the union of the $\p^{d_i}$'s spanning the
dual $\p^n$.
He also showed that the conjecture holds in the plane for up to $6$
points.


\vspace{0.2cm}

In this paper we study in detail the Torelli problem for the sheaf
$\tilde{\Omega}_{\p^n}(\log(D))$.
We denote by $Z$ a finite set of points, say $\ell$ points $z_1,
\ldots, z_\ell$, lying in
the dual space $\p_n$ of $\p^n$, and by $D_Z$ the union of the
corresponding hyperplanes $H_{z_1},\ldots,H_{z_\ell}$.
In order to state our result, we need to introduce what we call {\it Kronecker-Weierstrass
varieties} (a reason for this name will be apparent later on). 
If $(d,n_1,\ldots,n_s)$ is a string of $s+1$ integers such
that $n = d + n_1 \cdots + n_s$, we say that $Y \subset \p_n$ is a
{\it Kronecker-Weierstrass (KW) variety of type $(d;s)$} if
$Y = C \cup L_1 \cup \cdots \cup
L_s \subset \p_n$, where the $L_i$'s are linear
subspaces of dimension $1 \le n_i \le n-1$ and $C$ is a 
smooth rational curve of degree $d$, with $0 \leq d \leq n$ spanning
a linear space $L$ of dimension $d$ such that:
\begin{enumerate}[i)]
\item \label{I} 
  for all $i$,  $L \cap L_i$ is a single point which lies in $C$;
\item \label{II}
  the spaces $L_i$'s are mutually disjoint.
\end{enumerate}
In the case $d=0$ (so $C$ is reduced to a single point $y$), we
replace the conditions by the fact 
that all the linear spaces $L_i$ meet only at $y$.
The point $y$ in this case is called the distinguished point of $Y$.

We formulate now our main result. We give it here also for subschemes
with multiple structure, we will see how to make sense of this further on.
\begin{main} \label{main}
  Let $Z \subset \p_n$ be a finite-length, set-theoretically non-degenerate subscheme.

  Then $Z$ fails to be Torelli if and only if $Z$
  is contained in a KW variety $Y \subset \p_n$ of type
  $(d;s)$ whose distinguished point (for $d=0$) does not lie in $Z$.
\end{main}

The main ingredient that we bring in the proof is a functorial
definition of $\tilde{\Omega}_{\p^n}(\log D_Z)$ as the dualized direct
image of the sheaf of linear forms vanishing at $Z$ in $\p_n$, under
the natural point-hyperplane incidence variety. The key point is that this has to be taken
with a grain of salt, namely all functors have to be derived in order
to make the correspondence work smoothly.


As a corollary of the theorem above, we get that if $Z$ is contained
in a stable rational curve in $\p_n$, then $Z$ is not Torelli, as
conjectured by Dolgachev.

As another corollary, we will see that the converse implication holds
on $\p^2$, even without the assumption that
$\tilde{\Omega}_{\p^n}(\log D_Z)$ is semistable. 
In higher dimension, this implication no longer holds, 
regardless of $\tilde{\Omega}_{\p^n}(\log D_Z)$ being semistable or not.
To understand why, one first remarks that in many examples
$Z$ is contained in
a KW variety $Y$ without lying on a stable rational curve.
Yet one has to prove semistability of $\tilde{\Omega}_{\p^n}(\log
D_Z)$ for some of these examples.
One way to do this is to provide a filtration of $\tilde{\Omega}_{\p^n}(\log D_Z)$
associated to the decomposition of $Y$ into irreducible components.
This is the content of Theorem \ref{thm:decomposition}.
Some exceptions to the ``if'' direction of Dolgachev's conjecture
are Example \ref{NO} and \ref{NO2}.


\subsection{Structure of the paper} In the next section we set up our
framework for dealing with logarithmic sheaves, based on direct images
of ideal sheaves. In section \ref{section-unstable} we prove our main
theorem, already stated above. This section also contains a result on
the maximal number of unstable hyperplanes of a Steiner sheaf, see
Theorem \ref{max}.
Section \ref{section-decomposition} is
devoted to build a decomposition tool for non-Torelli arrangements.
In this last section we will outline some examples with interesting non-Torelli phenomena.

\subsection{Notations}

We refer to \cite{orlik-terao:arrangements-hyperplanes} for basic
notions on hyperplane arrangements.
As a matter of notation,
we let $\p^n$ be the space of $1$-dimensional quotients of a $\kk$-vector
space $V$ of dimension $n+1$ over a field $\kk$, and we write $\p^n = \p(V)$.
We let $\p_n = \p(V^*)$ be the dual of $\p^n$, namely
the space of hyperplanes of $\p^n$. Given a point $y \in \p_n$, we let
$H_y$ be the hyperplane of $\p^n$ given by $y$.
We use the variables $x_0,\ldots,x_n$ for the polynomial ring of
$\p^n$,
and the variables $z_0,\ldots,z_n$ for the polynomial ring of $\p_n$.

Let $Z$ be a finite length subscheme of the dual space $\p_n$ of $\p^n$.
The scheme $Z$ consists of finitely many points $y_1,\ldots,y_s$,
each $y_i$ supporting a subscheme of length $m_i$. Then $Z$ defines the
divisor $D_Z$ in $\p^n$, namely the set $H_{y_1},\ldots,H_{y_s}$ 
of hyperplanes of $\p^n$, each $H_{y_i}$ counted with multiplicity
$m_i$. Namely:
\[
D_Z = m_1H_{y_1}+ \cdots + m_s H_{y_s}.
\]

We will have to deal with complexes of coherent sheaves on $\p^n$.
A natural framework for them is the derived category $\Db(\p^n)$ of complexes of sheaves with bounded coherent cohomology.
We refer to \cite{gelfand-manin:homological} for a comprehensive
treatment.
We will denote by $[i]$ the $i$-th shift to the right of a complex in
the derived category.
To shorten notations, we will denote by $(a \to b \to c \shift)$
the exact triangle $(a \to b \to c \to a[1])$.
We will write $\RR F$ for the right derived functor of a functor $F$,
with image in the derived category.


\section{The Steiner sheaf associated to a hyperplane arrangements}

We consider the incidence variety $\F_n^n$ of pairs $(x,y) \in \p^n
\times \p_n$ where $x$ lies in $H_y$. We let $p$ and $q$ be the
projections from $\F_n^n$ respectively to $\p^n$ and to $\p_n$.
These projections are $\p^{n-1}$-bundles.
We have the natural exact sequence:
\begin{equation}
  \label{incidence}
  0 \to \sO_{\p^n \times \p_n}(-1,-1) \to \sO_{\p^n \times \p_n} \to \sO_{\F_n^n}
  \to 0.
\end{equation}

We consider the complex $\RR p_*(q^*(\cI_Z(1)))$ as an element of the
derived category of complexes of coherent sheaves on $\p^n$.
We set here the definition of a sheaf $\sF_Z$ on $\p^n$ attached to
$Z$, although  
it will turn out (Proposition \ref{iso}) that $\sF_Z$ is in fact
isomorphic to the sheaf 
$\tilde{\Omega}_{\p^n}(\log D_Z)$ introduced by Dolgachev.
However we will stick to the shorter notation $\sF_Z$ all over the paper.

\begin{dfn}
  Given a finite length subscheme $Z$ of $\p_n$ we define
  \[
  \sF_Z = \RRHHom_{\p^n}(\RR p_*(q^*(\cI_Z(1))),\sO_{\p^n}(-1)).
  \]
\end{dfn}
Whenever the vector space $V$ underlying $\p^n$ is unclear, we will
rather write $\sF^V_Z$.

\begin{prop} \label{torsion}
  Let $Z \subset \p_n$ be (schematically) non-degenerate subscheme of
  length $\ell$. 
  Then $\sF_Z$ is a sheaf having the following resolution:
  \[
  0 \to \sO_{\p^n}(-1)^{\ell-(n+1)} \to \sO_{\p^n}^{\ell-1} \to \sF_Z \to 0.
  \]
  Moreover, $\sF_Z$ is torsion-free if, locally around any point $z \in Z$, 
  we have $\cI^2_z \subset \cI_Z$.
\end{prop}


\begin{proof}
  Working on the product $\p^n \times \p_n$, we tensor
  \eqref{incidence} with $q^*(\cI_Z(1))$,  
  obtaining thus the exact sequence:
  \begin{equation}
    \label{product}
  0 \to \sO_{\p^n}(-1) \boxtimes \cI_Z \to \sO_{\p^n}
  \boxtimes \cI_Z(1) \to q^*(\cI_Z(1)) \to 0 
  \end{equation}

  Since $Z$ has finite length, we have $\HH^k(\p_n,\cI_Z(t))=0$ for all $k>1$
  and for all $t\in \Z$. Further, we have $\HH^0(\p_n,\cI_Z)=0$ for
  $Z$ is not empty and $\HH^0(\p_n,\cI_Z(1))=0$ since $Z$ is non-degenerate.
  Therefore, taking direct image onto $\p^n$, we get 
  the following distinguished triangle:
  \[
  \RR p_*(q^*(\cI_Z(1))) \to \sO_{\p^n}(-1)^{\ell-1} \xr{M_Z}
  \sO_{\p^n}^{\ell-(n+1)} \to \RR p_*(q^*(\cI_Z(1)))[1]
  \]
  where $M_Z$ is obtained applying $\RR p_*(-)$ to the inclusion
  appearing in \eqref{product}.
  Therefore $\RR p_*(q^*(\cI_Z(1)))$ has cohomology only in degree $0$
  and $1$, and is isomorphic to the
  cone of:
  \[\sO_{\p^n}(-1)^{\ell-1} \xr{M_Z} \sO_{\p^n}^{\ell-(n+1)}.\]
  Taking $\RRHHom_{\p^n}(-,\sO_{\p^n}(-1))$, we get that $\sF_Z$ is
  isomorphic to the cone of:
  \[
  \sO_{\p^n}(-1)^{\ell-(n+1)} \xr{M_Z^t} \sO_{\p^n}^{\ell-1}.
  \]
  Further, the sheaf $\RR^1 p_*(q^*(\cI_Z(1)))$ is supported at the
  points $x$ of $\p^n$ such that $\HH^1(H_x,\cI_{Z \cap H_x}(1))\neq 0$.
    In particular, it is a torsion sheaf.
  Therefore, the map $M_Z^t$ is injective, hence $\sF_Z$ is
  concentrated in degree zero, and we have the exact sequence: 
  \begin{equation}
    \label{M}
  0 \to \sO_{\p^n}(-1)^{\ell-(n+1)} \xr{M_Z^t} \sO_{\p^n}^{\ell-1} \to \sF_Z \to 0.    
  \end{equation}

  It remains to prove that $\sF_Z$ is torsion-free under our assumptions.
  Unwinding the double complex
  $\RRHHom_{\p^n}(\RR p_*(q^*(\cI_Z(1))),\sO_{\p^n}(-1))$, we get
  two short exact sequences:
  \begin{gather} 
    \label{torsion-1} 0  \to  \cExt^1_{\p^n}(\RR^1p_* q^*(\cI_Z(1)),\sO_{\p^n}(-1))  \to
    \sF_Z \to \sK \to 0, \\
    \label{torsion-2}  \sK \mono  \cHom_{\p^n}(p_* q^*(\cI_Z(1)),\sO_{\p^n}(-1)) \to
    \cExt^2_{\p^n}(\RR^1p_* q^*(\cI_Z(1)),\sO_{\p^n}(-1)) \to 0.
  \end{gather}

  The coherent sheaf $\sK$ is always torsion-free, and it differs from
  $\sF_Z$ if and only if $\RR^1p_* q^*(\cI_Z(1))$ is supported in
  codimension $1$.
  A necessary and sufficient condition for $\RR^1p_*
  q^*(\cI_Z(1))$ to be supported in 
  codimension $1$, is that there is $z \in Z$ such that,
  for all $x \in H_z$, we have $\HH^1(H_x,\cI_{Z \cap H_x}(1))\neq 0$.
  This is equivalent to say that, given any linear form $f$ vanishing
  at $z$, the ideal of $Z$ modulo $f$ contains all the quadrics of $R/f$.

  In order to check the above condition, we can assume that
  the reduced support of $Z$ is a single point, for $H_x$ generically
  avoids all other points.
  Working locally around this point $z \in Z$, our hypothesis is thus that
  all quadrics of vanishing at $z$ are in the ideal of $Z$.
  Therefore, the same thing takes place modulo $f$, and we are done.
\end{proof}

Let us describe briefly the relationship between our sheaf $\sF_Z$ and
the sheaves $\Omega_{\p^n}(\log D_Z)$ and $\tilde{\Omega}_{\p^n}(\log D_Z)$.
First, let us recall a definition of  $\Omega_{\p^n}(\log D_Z)$ (we
refer for instance to \cite{schenck:modifications}). Let $f$ be a
polynomial defining $D_Z$, where $Z$ consists of $\ell$ points of
$\p_n$. We consider the sheafified derivation module  $\mathscr{D}_0(Z)$, defined by the exact sequence:
\begin{equation}
  \label{gradient}
0 \to \mathscr{D}_0(Z) \to \sO_{\p^n}^{n+1} \xr{(\partial_0 f, \ldots, \partial_n f)} 
\sO_{\p^n}(\ell -1).
\end{equation}
Then the sheaf  $\Omega_{\p^n}(\log D_Z)$ is defined as:
\[
\Omega_{\p^n}(\log D_Z) = \cHom_{\p^n}(\mathscr{D}_0(Z),\sO_{\p^n}(-1)).
\]

\begin{prop} \label{iso}
  Assume that $Z$ is reduced and non-degenerate. Then $\sF_Z$ is isomorphic
  to Dolgachev's sheaf $\tilde{\Omega}_{\p^n}(\log D_Z)$.
  Moreover, we have:
  \begin{equation}
    \label{duali}
  \Omega_{\p^n}(\log D_Z) \cong \cHom_{\p^n}(p_*
  q^*(\cI_Z(1)),\sO_{\p^n}(-1)) \cong \sF_Z^{**}.
  \end{equation}
\end{prop}

\begin{proof}
  Let us first prove the claim regarding $\tilde{\Omega}_{\p^n}(\log D_Z)$.
We apply the functor $\RR p_* q^*$ to the exact sequence:
  \[
  0 \to \cI_Z(1) \to \sO_{\p_n}(1) \to \sO_Z \to 0.
  \]
  Using \eqref{product}, we obtain the distinguished triangle:
  \begin{equation}
    \label{OZ}
  \RR p_*(q^*(\cI_Z(1))) \to \cT_{\p^n}(-1) \to 
   \RR p_*(q^*(\sO_Z)) \shift
  \end{equation}
  Now if $Z$ is reduced we have $Z = \{z_1,\ldots,z_\ell\}$. Note that:
  \[
  q^*(\sO_Z) \cong \sO_{q^{-1}(Z)} \cong \sO_{\cup_{j = 1,\ldots,\ell} H_{z_j}}.
  \]
  This sheaf 
 lies above the divisor $D_Z$, and
  $p: q^{-1}(Z) \to D_Z$
  is a resolution of singularities of $D_Z$.
  By Grothendieck duality, we have that $\RRHHom_{\p^n}(\RR
  p_*(q^*(\sO_Z)),\sO_{\p^n}(-1))[1]$ is isomorphic to:
  \begin{align*}
  & \RR p_*(\RRHHom_{\F^n_n}(q^*(\sO_Z),\sO_{\F^n_n}(0,-n)))[n].
  \end{align*}
  For each $z_j$ in $Z$ we have:
  \begin{align*}
    \RRHHom_{\F^n_n}(q^*(\sO_{z_j}),\sO_{\F^n_n}(0,-n))[n]
    & \cong  \cExt_{\F^n_n}^{n}(\sO_{H_{z_j}} ,\sO_{\F^n_n}(0,-n)) ) \cong  \\
  & \cong  \sO_{H_{z_j}} \ts \omega^*_{\F^n_n} \ts \sO_{\F^n_n}(0,-n)
  \cong \sO_{H_{z_j}}.
  \end{align*}

  Therefore, taking $\RRHHom_{\p^n}(-,\sO_{\p^n}(-1))$ of the triangle
  \eqref{OZ}, we 
  have the exact sequence: 
  \begin{equation}
    \label{extension}
    0 \to \Omega_{\p^n} \to \sF_Z \to   p_*(\sO_{q^{-1}(Z)}) \to 0.    
  \end{equation}

  We will be done if we can prove that this is the residue exact
  sequence defining 
  $\tilde{\Omega}_{\p^n}(\log D_Z)$ according to
  \cite{dolgachev:logarithmic}.
  This will be accomplished by proving that there is in fact a unique
  functorial extension of $\Omega_{\p^n}(1)$ by
  $p_*(q^*(\sO_Z))$, and observing that both the residue exact sequence and
  \eqref{extension} are clearly functorial.
  \begin{claim} \label{theclaim}
    We have a natural isomorphism:
    \[
    \Ext^1_{\p^n}(p_*(q^*(\sO_Z)),\Omega_{\p^n}) \cong 
    \Hom_{\p_n}(\sO_{\p_n},\sO_Z)^*.
    \]
  \end{claim}
  Since $\sO_Z$ is naturally a quotient of $\sO_{\p_n}$, this claim
  will complete our argument.
  To prove the claim, we write the isomorphisms:
  \begin{align*}
    \Ext^1_{\p^n}(p_*(q^*(\sO_Z)),\Omega_{\p^n}) & \cong 
    \Ext^{n-1}_{\p^n}(\Omega_{\p^n}(n+1),p_*(q^*(\sO_Z)))^* \cong \\
    & \cong \Ext^{n-1}_{\F^n_n}(p^*(\Omega_{\p^n}(n+1)),q^*(\sO_Z))^*,
    \intertext{where the first one is Serre duality and the second
      one is adjunction. 
  Now we use the left adjoint functor to $q^*$, namely the functor
  $\RR q_*(- \ts \sO_{\F^n_n}(-n,1))[n-1]$. Thus the latter group above is}
    & \cong \Hom_{\p_n}(\RR q_*(p^*(\Omega_{\p^n}(1)))\ts
    \sO_{\p^n}(1),\sO_Z)^*    \cong \\ 
    & \cong \Hom_{\p_n}(\sO_{\p_n},\sO_Z)^*.
  \end{align*}

Let us now turn to $\Omega_{\p^n}(\log D_Z)$. Let again $f=\prod_{i=1}^\ell f_i$ be an
equation defining $D_Z$. Recall that the image of
the rightmost map in \eqref{gradient} (the gradient map) is the
Jacobian ideal $\sJ$ of $D_Z$. Denote by $\sJ_{D_Z}$ the image of
$\sJ$ in $\sO_{D_Z}$ (so $\sJ_{D_Z} = \sJ \cdot \sO_{D_Z}$).
Recall the natural exact sequence relating $\sJ_{D_Z}$ and $\mathscr{D}_0(Z)$
(see e.g. \cite[Section 2]{dolgachev:logarithmic}):
\begin{equation}
  \label{vai}
 \xymatrix@-2ex{
 0  \ar[r] &  \mathscr{D}_0(Z) \ar[r] &  \cT_{\p^n}(-1) \ar[r] &  \sJ_{D_Z}(\ell -1) \ar[r] &  0.
 }
\end{equation}
Note also that we have:
\begin{align*}
\Ext^1_{\p^n}(p_*q^*(\sO_Z),\Omega_{\p^n}) & \cong
\Hom_{\p^n}(\cT_{\p^n},\RRHHom_{\p^n}(p_*q^*(\sO_Z),\sO_{\p^n})[1]) \cong \\
& \cong \Hom_{\p^n}(\cT_{\p^n},p_*q^*(\sO_Z)(1)),
\end{align*}
so the last homomorphism group contains a unique functorial element.
Further, from \cite[Proposition 2.4]{dolgachev:logarithmic} we get an
inclusion of $\sJ_{D_Z}(\ell)$ into $p_*(\omega_{q^{-1} Z} \ts \omega^*_{\p^n}) \cong
p_*(q^* \sO_Z)(1)$.

Therefore, both $\mathscr{D}_0(Z)$ (by \eqref{vai}) and
$p_*(q^*(\cI_Z(1)))$ (by the cohomology sequence of \eqref{OZ}) are
the kernel of the unique functorial map $\cT_{\p^n}(-1)
\to p_*q^*(\sO_Z)$.
This gives an isomorphism:
\begin{equation}
  \label{last}
  p_*(q^*(\cI_Z(1))) \cong \mathscr{D}_0(Z).  
\end{equation}
Note also that we have the exact sequence:
\[
0 \to \sF_Z \to \Omega_{\p^n}(\log D_Z) \to
\cExt^2_{\p^n}(\RR^1p_* q^*(\cI_Z(1)),\sO_{\p^n}(-1))  \to 0.
\]
The desired isomorphisms \eqref{duali} easily follow from the above
sequence and \eqref{last}.
\end{proof}

\begin{rmk}
  \label{omegatilde-omega}
  The support of the cokernel sheaf $\cExt^2_{\p^n}(\RR^1p_*
  q^*(\cI_Z(1)),\sO_{\p^n}(-1))$ sits in codimension $k>1$ if and only
  $Z$ contains a subscheme of length $(n+1)$, contained in a
  linear subspace $\p_{k-1}$.
  Further, this shows again that $\sF_Z$ and $\Omega_{\p^n}(\log D_Z)$ agree if
  $D_Z$ is normal crossing in codimension $2$, see \cite[Corollary 2.8]{dolgachev:logarithmic}.
\end{rmk}

\begin{eg}
  Consider the ideal $(z_0z_2^2,(z_1+z_1)z_1z_2,z_0z_1z_2,z_0z_1^2)$. This
  defines a subscheme $Z \subset \p_2$, which is the union of the first
  infinitesimal 
  neighbourhood of $(1:0:0)$ and the three collinear points $(0:1:0)$,
  $(0:0:1)$, $(0:1:-1)$. Then we have:
  \[
  M_Z=\left (
 \begin{array}{ccccc}
   -x_0 & 0 & x_1 & 0 & 0 \\
   x_0 & 0 & 0 & x_1-x_2 & -x_2\\
   0 & x_0 & 0  & 0 & x_2
 \end{array} \right ).
 \]
 In this case $\sF_Z$ is still torsion-free and we have:
 \[
 0 \to \sF_Z \to \Omega_{\p^n}(\log D_Z) \to \sO_{x_1,\ldots,x_4}
 \to 0, \qquad \Omega_{\p^n}(\log D_Z) \cong \sO_{\p^2}(2) \oplus \sO_{\p^2}(1),
 \]
 where $x_1,\ldots,x_4$ are $(1:0:0)$, $(0:1:0)$, $(0:0:1)$,
 $(0:1:1)$,  the $4$ points
 corresponding to the $4$ lines in $\p_2$ which are $3$-secant to $Z$.
 The arrangement given by $Z$ is thus free (i.e. $\Omega_{\p^n}(\log
 D_Z)$ splits as a direct sum of line bundles).
\end{eg}

\begin{eg}
  Consider the scheme $Z$ defined as the union of the second
  infinitesimal neighbourhood of $(0:1:0)$ and the two points $(1:0:0)$,
  $(0:0:1)$. Namely, the ideal of $Z$ is
  $\text{${\text{ideal}}$} ({x}_{0} {x}_{2}^{2},{x}_{0}^{2} {x}_{2},{x}_{1} {x}_{2}^{3},{x}_{0}^{3} {x}_{1})$.
  In this case, we obtain the matrix:
  \[
  M_Z=\left (
 \begin{array}{cccccccc}
   0 & 0 & -x_1 & 0 & 0 & 0 & x_2 \\
   x_0 & 0 & 0 & x_1 & 0 & 0 & 0\\
   0 & 0 & x_2 & 0 & x_1 & 0 & 0 \\
   -x_1 & x_0 & 0  & 0 & 0 & 0 & 0 \\
   -x_2 & 0 & x_0 & 0  & 0 & x_1 & 0
 \end{array} \right ).
 \]

 Here we get the line $L$ defined as $\{x_1=0\}$ as support of the torsion part of
 $\sF_Z$. 
 We have $\Omega_{\p^n}(\log D_Z) \cong \sO_{\p^2}(2) \oplus
 \sO_{\p^2}(2)$, i.e. $Z$ is a free arrangement.
 The exact sequences \eqref{torsion-1} and \eqref{torsion-2} become:
 \[
 0 \to \sO_L(-2) \to \sF_Z \to \sO_{\p^2}(2)^2 \to \sO_{Z_1
   \cup Z_2} \to 0,
 \]
 where $Z_1,Z_2$ are two length-$2$ subschemes, supported at
 the points $(1:0:0)$ and $(0:0:1)$, accounting for the two $4$-secant lines to $Z$
 in $\p_2$, namely $\{z_0=0\}$ and $\{z_2=0 \}$.
\end{eg}


\section{Unstable hyperplanes of logarithmic sheaves}
\label{section-unstable}

The goal of this section is to prove our main result,
stated in the introduction.
We will first need some definitions.

\begin{dfn}
  Let $\sE$ be a Steiner sheaf on $\p^n$, namely a sheaf $\sE$ fitting
  into an exact sequence of the form:
  \[
  0 \to \sO_{\p^n}(-1)^a \to \sO_{\p^n}^b \to \sE \to 0,
  \]
  for some integers $a,b$.
  Then a hyperplane $H$ is {\it unstable} for $\sE$ if:
  \[
  \HH^{n-1}(H,\sE_{|H}(-n)) \neq 0.
  \]
  A point $y$ of $\p_n$ is unstable for $\sE$ if the hyperplane $H_y$
  is unstable for $\sE$.
    
  We can give a scheme structure to the set $\W(\sE)$ of unstable
  hyperplanes of $\sE$,
  considering them as the scheme-theoretic support of the sheaf $\RR^{n-1}
  q_*(p^*(\sE(-n)))$.
\end{dfn}

\begin{dfn}
  A finite length subscheme $Z$ of $\p_n$ is said to be {\it Torelli} if $Z$
  gives rise to a Torelli arrangement, namely if the {\it set} of unstable
  hyperplanes of $\sF_Z$ is the support of $Z$, i.e. if we have a
  set-theoretic equality:
  \[
  \W(\sF_Z)=Z.
  \]
\end{dfn}

\begin{lem}
  Let $Z$ be a finite length subscheme of $\p_n$.
  Then we have a scheme-theoretic inclusion:
  \[
  Z \subset \W(\sF_Z).
  \]
\end{lem}

\begin{proof}
  By Grothendieck duality, we have:
  \[
  \sF_Z(-n) \cong \RR p_* (\RRHHom_{\F^n_n}(q^*(\cI_Z(1)),\sO_{\F^n_n}(-n,-n)))[n-1],
  \]
  from which we get an epimorphism:
  \[
  \RR q_*(p^*(\sF_Z(-n)))[n-1] \to \RR q_* (\RRHHom_{\F^n_n}(q^*(\cI_Z(1)),\sO_{\F^n_n}(-n,-n)))[n-1].
  \]
  Applying again Grothendieck duality, we get an isomorphism of the
  right hand side above and:
  \[
  \RRHHom_{\p_n}(\RR q_* q^*(\cI_Z(1)),\sO_{\p_n}(-n)),
  \]
  which projects onto:
  \[
  \RRHHom_{\p_n}(\cI_Z(1),\sO_{\p_n}(-n)).
  \]
  Summing up, we have an epimorphism:
  \[
  \RR q_*(p^*(\sF_Z(-n)))[n-1] \epi  \RRHHom_{\p_n}(\cI_Z(1),\sO_{\p_n}(-n)),
  \]
  and taking cohomology in degree $n-1$ we get:
  \[
  \RR^{n-1} q_*(p^*(\sF_Z(-n))) \epi \cExt^{n-1}_{\p_n}(\cI_Z,\sO_{\p_n}(-n-1)) \cong \sO_Z,
  \]
  which proves our claim.
\end{proof}

\begin{rmk}
  It was already proved in \cite{dolgachev:logarithmic} that any $z \in Z$ is
  unstable for $\sF_Z$, hence $Z$ is not Torelli if and only if 
  the {\it set} of unstable hyperplanes of $\sF_Z$ strictly contains $Z$.
  
  One could say that $Z$ is {\it scheme-theoretically Torelli} if the
  subscheme of unstable hyperplanes is $Z$ itself.
  A criterion analogous to Theorem \ref{main} for $Z$ to be
  scheme-theoretically Torelli is lacking at 
  the time being.
\end{rmk}

\begin{rmk}
We point out that $\W(\sF_Z)=\W(\Omega_{\p^n}(\log D_Z))$ if and only if
$Z$ does not possess a subscheme of length $(n+1)$ contained in a line,
as explained in Remark \ref{omegatilde-omega}.
This remark makes more precise Proposition 3.2 of \cite{dolgachev:logarithmic}.
\end{rmk}

\subsection{Kronecker-Weierstrass varieties and unstable hyperplanes}

In order to prove Theorem \ref{main}, we introduce
some geometric objects that we call Kronecker-Weierstrass varieties.
The name is inspired on the tool that classifies them.
Indeed, the isomorphism classes of these varieties are given by the
standard Kronecker-Weierstrass 
forms of a matrix of homogeneous linear forms in two variables.
We recall the definition given in the introduction.

\begin{dfn} \label{KW}
Let $(d,n_1,\ldots,n_s)$ be a string of $s+1$ integers such
that $n = d + n_1 \cdots + n_s$, and $1 \le d \le n$.
Then $Y \subset \p_n$ is a
{\it Kronecker-Weierstrass (KW) variety of type $(d;s)$} if
$Y = C \cup L_1 \cup \cdots \cup
L_s \subset \p_n$, where the $L_i$'s are linear
subspaces of dimension $1 \le n_i \le n-1$ and $C$ is a 
smooth rational curve of degree $d$ (called the
{\it curve part} of $Y$) spanning
a linear space $L$ of dimension $d$ such that:
\begin{enumerate}[i)]
\item 
  for all $i$,  $L \cap L_i$ is a single point which lies in $C$;
\item 
  the spaces $L_i$'s are mutually disjoint.
\end{enumerate}

If $d=0$, a KW variety of type $(0;s)$ is defined as 
$Y = L_1 \cup \cdots \cup L_s \subset \p_n$, 
where the $L_i$'s are linear
subspaces of dimension $1 \le n_i \le n-1$
and all the linear spaces $L_i$ meet only at a point $y$, 
which is called {\it the distinguished point of $Y$}.
\end{dfn}

\begin{figure}[h!]
  \centering
  \includegraphics[height=1.5in]{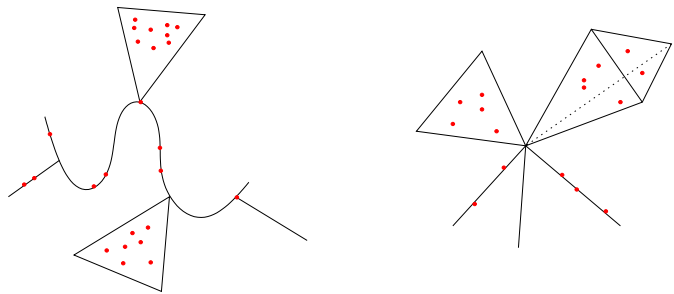}
  \caption{Points contained in a Kronecker Weierstrass variety.}
\end{figure}

  \begin{eg} We outline some examples of KW variety.
  \begin{enumerate}[1)]
  \item A rational normal curve is a KW variety of type $(n;0)$.
  \item A union of two lines in $\p^2$ is a KW  variety in three
     ways, two of them of type $(1;1)$, and one of type $(0;2)$ (the
     intersection point is the distinguished point).
  \end{enumerate}
  \end{eg}

Having this setup, we can move towards the proof of our Theorem \ref{main}.
We need a series of lemmas and the following construction.

  Given a point $y$ of $\p_n$, we consider the Koszul complex resolving the
  ideal sheaf $\cI_y$, namely a long exact sequence:
  \[
  0 \to \sO_{\p_n}(-n) \xr{d_n} \sO_{\p_n}^{n}(-n+1) \xr{d_{n-1}} \cdots \xr{d_3}
  \sO_{\p_n}^{n}(-2) \xr{d_2} \sO_{\p_n}(-1) \xr{d_1} \cI_y \to 0.
  \]
  We let $\cS_y$ be the sheaf $\im(d_{n-1})$, twisted by $\sO_{\p_n}(n)$.
  We have:
  \begin{equation}
    \label{S}
    0 \to \sO_{\p_n} \xr{(h_1,\ldots,h_n)} \sO_{\p_n}^{n}(1) \to \cS_y\to 0,  
  \end{equation}
  where the $h_i$'s are linear forms on $\p_n$ and $y$ is defined by
  $\{h_1=\cdots=h_n=0\}$.

\vspace{0.2cm}

The following lemma is the key to our argument.
It is inspired on a generalization of \cite[Proposition
6.1]{valles:schwarzenberger-doi}

\begin{lem} \label{sezione}
  Let $y$ be a point of $\p_n$, and let $Z$ be a finite length
  subscheme of $\p_n$ not containing $y$. Then $y$ is unstable for
  $\sF_Z$ if and only if $\HH^0(\p_n,\cS_y \ts \cI_Z)\neq 0$.
\end{lem}

\begin{proof}
  By definition $y$ is unstable for $\sF_Z$ if and only if:
  \[
  \HH^{n-1}(H_y,\sF_Z(-n)) \neq 0.
  \]
  In view of the exact sequence \eqref{M}, 
  this is equivalent to say that, restricting the matrix $M_Z^t$ to
  $H_y$ and taking cohomology, we get a 
  non-zero cokernel of: 
  \[
  \HH^{n-1}(H_y,\sO_{H_y}^{\ell-(n+1)}(-n-1)) \xr{(M_Z^t)_{| H_y}}
  \HH^{n-1}(H_y,\sO_{H_y}^{\ell-1}(-n)).
  \]
  By Serre duality, this means that 
  \[
  \HH^{0}(H_y,\sO_{H_y}^{\ell-1}) \xr{(M_Z)_{| H_y}}
  \HH^{0}(H_y,\sO_{H_y}^{\ell-(n+1)}(1)) 
  \]
  has non-trivial kernel.
  Recalling by the proof of Proposition
  \ref{torsion} that $\RR p_* (q^*(\cI_Z(1))$ is the cone of the map $M_Z$,
  we see that this is equivalent to say that:
  \[
  \Ext^1_{\p^n}(\sO_{H_y},\RR p_* (q^*(\cI_Z(1))) \neq 0.
  \]
  Since $(p^*,\RR p_*)$ is an adjoint pair, the above extension group is
  isomorphic to:
  \[
  \Ext^1_{\F^n_n}(p^*(\sO_{H_y}),q^*(\cI_Z(1))).
  \]
  
  We use again the left adjoint functor to $q^*$ (recall that it is
  $\RR q_*(- \ts \sO_{\F^n_n}(-n,1))[n-1]$). 
  The above group is thus isomorphic to:
  \begin{equation}
    \label{Ext}
  \Ext^{2-n}_{\p_n}(\RR q_* p^*(\sO_{H_y}(-n)),\cI_Z).    
  \end{equation}
  Note also that we can compute \eqref{Ext} as:
  \begin{equation}
    \label{H0}
  \HH^{\cdot}(\p_n,\RRHHom_{\p_n}\left(\RR q_* p^*(\sO_{H_y}(-n)[2-n]) ,
  \sO_{\p_n} \right)  \ts \cI_Z).
  \end{equation}

  Let us now compute $\RR q_* p^*(\sO_{H_y}(-n))$. Making use of
  \eqref{incidence},
  we get a distinguished triangle:
  \[
  \RR q_* p^*(\sO_{H_y}(-n)) \to \sO_{\p_n}(-1)^n[-n+2] \xr{P_y}
  \sO_{\p_n}[-n+2] \shift 
  \]
  Here, it is easy to see that $P_y$ is a matrix of linear forms
  defining $y$ in $\p_n$.
  Dualizing the above diagram, we get an exact sequence (of sheaves):
  \[
  0 \to \sO_{\p_n} \xr{P_y^t} \sO_{\p_n}(1)^n\to
  \RRHHom_{\p_n}\left(\RR q_* p^*(\sO_{H_y}(-n)),\sO_{\p_n}\right)[-n+2] \to 0.
  \]
  By the definition of the sheaf $\cS_y$, we have thus an isomorphism:
  \[
  \cS_y \cong \RRHHom_{\p_n}\left(\RR q_*
    p^*(\sO_{H_y}(-n)),\sO_{\p_n}\right)[-n+2].\]
  Then the space appearing in \eqref{H0} is non-zero if and only if
  \[
  \HH^{0}(\p_n, \cS_y \overset{\mathbf{L}}{\ts} \cI_Z) \neq 0,
  \]
  where the notation above stands for left-derived tensor product.
  But one easily proves that $\mathcal{T}or_j(\cS_y,\cI_Z)=0$ for
  $j>0$, so \eqref{H0} is non-zero if and only if
  \[
  \HH^{0}(\p_n, \cS_y {\ts} \cI_Z) \neq 0.
  \]
  So $y$ is unstable if and only if the above vector space is not
  zero, and the lemma is proved.
\end{proof}

\begin{lem} \label{sezione-KW}
  Let $y$ be a point and $Z$ be a  finite-length, non-degenerate
  subscheme of $\p_n$, not containing $y$.
  Then $\HH^0(\p_n,\cS_y \ts \cI_Z)\neq 0$ if and only if $Z$
  is contained in the rank-$1$ locus of a $2\times n$ matrix $M$ of linear forms
  having non-proportional rows, with one row defining $y$.
\end{lem}

\begin{proof}
Recalling the exact sequence \eqref{S} defining $\cS_y$, we let
  $h_1,\ldots,h_n$ be a regular sequence defining $y \in \p_n$,
  and we note that
  a section in $\HH^0(\p_n,\cS_y \ts \cI_Z)$ is given by a global
  section $s$ of $\cS_y$ such that $s$ vanishes along $Z$.
  In turn, $s$ lifts to $\tilde{s}$ as in the diagram:
  \[
  \xymatrix@C+3ex{
    &&& \sO_{\p_n} \ar^-{s}[d] \ar@{.>}_-{\tilde{s}}[dl]\\
    0 \ar[r] & \sO_{\p_n} \ar^-{(h_1,\ldots,h_n)}[r] &
    \sO_{\p_n}^{n}(1) \ar[r] & \cS_y \ar[r] & 0. 
  }
    \]
    Now $\tilde{s}$ is given by $(g_1,\ldots,g_n)$, where the $g_i$'s
    are linear forms and the row $(g_1,\ldots,g_n)$ is not proportional to
    $(h_1,\ldots,h_n)$.
    Then in order for $s$ to vanish on $Z$, we must have that $Z$
    is contained in the locus $Y$ cut by the $2\times 2$ minors of the
    matrix:
    \[
    M = 
    \begin{pmatrix}
      h_1 & \cdots & h_n \\
      g_1 & \cdots & g_n
    \end{pmatrix},
    \]
    Note that $Y$ is not all of $\p_n$, because the two rows of $M$ are
    not proportional.
    Since all the construction is reversible, the lemma is proved.
  \end{proof}

  \begin{lem} \label{lemma-KW}
    Let $Z$ be a finite-length, set-theoretically non-degenerate
    subscheme of $\p^n$ and $y \in \p_n$.
    Then the equivalent conditions of the previous lemma are satisfied if and only
    $Z$ is contained in a KW variety $Y$ of type $(d;s)$ with either
    $d>0$ and $y$ is in the curve part of $Y$, or $d=0$, and $y$ is the distinguished point
    of $Y$.
  \end{lem}

  \begin{proof}
    Let us assume that the conditions of the previous lemma are
    satisfied, and look for the KW variety $Y$.
    So let us consider the matrix $M$ given by the above lemma as a morphism of sheaves:
    \[
    \sO_{\p_n}(-1)^n \to \sO_{\p_n}^2.
    \]
    We have that $Z$ is contained in the rank-$1$ locus of $M$, hence
    in the support of the cokernel sheaf $\sT$ of the above map, hence
    in the image in $\p_n$ of the natural map $\p(\sT) \to \p_n$.

    The matrix $M$ can be written in coordinates as $M_{i,j} = \sum_{k=0}^{n}
    a_{i,j,k} z_k$ for some scalars $a_{i,j,k}$, with $i=0,1$ and $j=0,\ldots,n-1$.
    This gives a matrix $N$ of size $n \times (n+1)$, this time over
    $\kk[\xi_0,\xi_1]$, by: 
    \begin{equation}
      \label{N}
      N_{j,k} = \sum_{i=0,1} a_{i,j,k} \xi_i.      
    \end{equation}
    Therefore, we think of the above matrix $N$ as a map:
    \begin{equation}
      \label{P1}
    N : \sO_{\p^1}(-1)^{n} \to \sO_{\p^1}^{n+1},
    \end{equation}
    where the target space is identified with $V  \ts \sO_{\p^1}$,
    with $V = \HH^0(\p_n,\sO_{\p_n}(1))$.

    Note that this map is injective.
    Indeed, if $y$ is defined by the forms $h_1,\ldots,h_n$, up to a change of basis we may assume
    $h_i = z_i$, so that the identity matrix of size $n$ is a
    submatrix $N$ evaluated at $(1:0)$.
    The sheaf $\sL = \cok(N)$ decomposes 
    as:
    \begin{equation}
      \label{decomposition}
    \sL  \cong \sO_{\p^1}(d) \oplus \sO_{\p^1,p_1}^{n_1}\oplus \cdots
    \oplus \sO_{\p^1,p_s}^{n_s},
    \end{equation}
    for some distinct points $p_i \in \p^1$, and some integers
    $d,n_1,\ldots,n_s \in [0,n]$. 
    Since the sheaf $\sL$ has degree $n$, we must have $d + n_1 + \cdots +
    n_s=n$.

    The matrix $N$ is classified by its standard Kronecker-Weierstrass (KW)
    form (hence the name of $Y$); we refer for this standard form for
    instance to \cite[Chapter 19]{burgisser-clausen-shokrollahi}.
    This means that $N$ can be written, in an appropriate basis, in
    block form like:
    \begin{equation} \label{block}
    N = \left(
      \begin{array}{c|c|c|c}
        N_0 & 0 & \cdots & 0 \\
        \hline 
        0 & N_1 & & 0 \\
        \hline 
        \vdots && \ddots \\
        \hline 
        0 & 0 & & N_s \\
      \end{array} \right).
    \end{equation}
    Here, $N_0$ is of size $d \times (d+1)$,
    with $\cok(N_0) \cong \sO_{\p^1}(d)$ and 
    $N_i$ is a square matrix of size $n_i$ that degenerates on $p_i$
    only. 
    For $i>0$, each $N_i$ can be further decomposed into its normal Jordan
    blocks, which are all of size one if and only if $N_i$ is diagonal.
     Note also that $N_0$ can be written as:
     \begin{equation}
       \label{N0}
    N_0 = \left(
      \begin{array}{ccccc}
        \xi_0 & 0 & \cdots & 0 \\
        \xi_1 & \xi_0 & \ddots & \vdots  \\
        0 & \ddots & \ddots & 0 \\
        \vdots & \ddots & \xi_1 & \xi_0 \\
        0 & \cdots  & 0 & \xi_1 \\
      \end{array} \right).
     \end{equation}

    Let us show that, with these elements, one can define $Y$.
    \begin{description}
    \item[Case $d>0$] 
    In this case, since $d + n_1 + \cdots +
    n_s=n$, we have $1 \le n_j \le n-1$ for all $j$.
    We define then the curve $C$ as the image of $\p(\sL)$ in $\p_n$
    obtained by taking global sections of the quotient $\sO_{\p^1}(d)$
    of $\sL$.
    Namely, $C$ is just $\p^1$ mapped to $\p_n$ by $\sO_{\p^1}(d)$,
    and spans the $d$-dimensional linear subspace $L =
    \p(\HH^0(\p^1,\sL))$ corresponding to the projection 
    $\HH^0(\p^1,\sL) \to \HH^0(\p^1,\sO_{\p^1}(d))$.
    In an appropriate basis, the curve $C$ is cut in the space $L = \{z_{d+1}=\cdots=z_n=0\}$
    as the rank-$1$ locus of:
        \[
    \begin{pmatrix}
      z_1 & \cdots & z_d \\
      z_0 & \cdots & z_{d-1}
    \end{pmatrix}.
    \]
    We define then $L_j$ as the cone over
    the image in $\p_n$ of $p_j$ and the space given by 
    the projection $\HH^0(\p^1,\sL) \to \HH^0(\p^1,\sO_{\p^1,p_j}^{n_j})$.
    Each $L_j$ meets $L$ only at $p_j$, and the $p_j$'s are all
    distinct if $d>0$. Since $L_i$ meets $L_j$ only along $C$, all
    linear spaces $L_j$'s are mutually disjoint for $d>0$.
    This defines the KW variety $Y = C \cup L_1 \cup \cdots \cup C_s$.

    Note that $y$ belongs to $C$. Indeed, in the basis under
    consideration, we have that $y=(1:0:\ldots:0)$, and $C$ goes
    through this point.
    Note also that $Y$ clearly contains the image 
    of $\p(\sL) \cong \p(\sT)$ in $\p_n$ under the natural map $\p(\sL) \to
    \p(\HH^0(\p^1,\sL))$. But this image is the
    rank-$1$ locus of $M$, which contains $Z$. So $Y$ contains $Z$.

    \item[Case $d=0$]
    In this case, under the decomposition \eqref{block}, we have $N_0=0$.
    The sheaf $\sL$ defines a projection of $\p^1$ to a point
    of $\p_n$, which in the basis under consideration has
    coordinates $(1:0:\ldots:0)$, i.e. this point is $y$.
    In this case, each linear space $L_j$ is a cone
    over $y$ and $\p(\HH^0(\p^1,\sO_{\p^1,p_j}^{n_j}))$, hence all
    the $L_j$'s meet only at $y$.
    Once we prove that $1 \le n_j \le n-1$
    for all $j$, we can define $Y = L_1\cup \cdots \cup L_s$,
    and clearly $Z$
    is contained in $Y$.

    So let us show $1 \le n_j \le n-1$
    for all $j$, in other words let us prove $s \ge 2$.
    Assume thus $s=1$, and note that $\sL \cong \sO_{\p^1,p}^{n}$,
    with $p_1 = p = (a:b)$, so that $N_1$ degenerates on $(a:b)$
    only.
    Note that the standard KW form of $N_1$ cannot be a
    multiple of the $n\times n$ identity matrix, times $b\xi_0 - a\xi_1$,
    for the two rows of the 
    corresponding matrix $M$ would be proportional.
    Hence the KW form of $N_1$ has at least one non-trivial
    Jordan block (i.e. of size at least $2$). Then, the corresponding
    rank-$1$ locus of $M$ is a 
    multiple structure over a linear space of dimension at most
    $n-1$. But then $Z$ is contained in a multiple structure over a hyperplane, a
    contradiction, since $Z$ is set-theoretically non-degenerate.
    \end{description}

    To prove the converse implication, let us be given a KW variety
    $Y$ of type $(d;s)$ containing $Z$, with $d>0$, let $L_0$ be the
    span of the curve part $C$ of $Y$ and let $L_1,\ldots,L_s$ the linear spaces of $Y$.
    For each $L_i$, we choose a basis of an $(n_i-1)$-dimensional
    linear subspace disjoint from $L_0$, and we complete this to a
    basis of $V$ by stacking a basis of $L_0$.
    We take
    $N_0$ as in \eqref{N0}, and, for $i=1,\ldots,s$, we let $(a_i,b_i)$ be the points on
    $\p¹$ corresponding to the intersection $C\cap L_i$ under the
    parametrization $\p¹ \to C$.
    We define $N_i$ as 
    a square matrix of size $n_i$
    having $b_i\xi_0 -
    a_i\xi_1$ on the diagonal and zero anywhere else.
    We have thus presented the matrix as is \eqref{N}, hence we have
    a $2 \times n$ of the form $M_{i,j} = \sum_{k=0}^{n}
    a_{i,j,k} z_k$ in the coordinates given by the chosen basis.
    The first row of $M$ thus defines $y$, and the rank-$1$ locus of
    $M$ is $Y$.

    If $d=0$ we choose a projection $\p^1 \to \{ y \}$, and we choose
    $s$ distinct points $(a_i:b_i)$ in $\p^1$.
    We still have the
    matrices $N_i$, and the matrix $N_0$ is the zero matrix with one row.
    Constructing $N$ as in \eqref{block}, the same choice of basis for $V$
    allows to write the matrix $M$, whose first row defines $y$ and
    whose rank-$1$ locus is $Y$.
  \end{proof}

  We can now prove our main result, Theorem \ref{main}.
  Namely, let $Z \subset \p_n$ be a finite-length,
  set-theoretically non-degenerate subscheme.
  Then we have to prove that the set of unstable hyperplanes
  $\W(\sF_Z)$ contains at least another point $y \not\in Z$ if and only if $Z$ is
  contained in a KW variety $Y$ of type $(d;s)$ whose distinguished
  point (if $d=0$) does not lie in $Z$.

  \begin{proof}[Proof of Theorem \ref{main}]
    Let us assume that $Z$ is not Torelli, and prove that $Z$ is
    contained in a KW variety.
    Since $Z$ is not Torelli, there is a point
    $y \in \p_n$, not belonging to $Z$, unstable for $\sF_Z$.
    We can apply Lemmas \ref{sezione}, \ref{sezione-KW},
    \ref{lemma-KW} since $Z$ is set-theoretically non-degenerate.
    Then, there is
    a KW variety $Y$ containing $Z$, and we are done.

    Conversely given a KW variety $Y$ of type $(d;s)$ containing $Z$,
    we look at two cases.
    If $d=0$, then by assumption $Z$ does not
    contain the distinguished point $y$ of $Y$. But by 
    Lemmas \ref{sezione}, \ref{sezione-KW},
    \ref{lemma-KW}, the point $y$ is unstable for $\sF_Z$, so $Z$ is
    not Torelli.
    If $d>0$, we let $y$ be any point of the curve part $C$ of $Y$. By 
    Lemmas \ref{sezione}, \ref{sezione-KW},
    \ref{lemma-KW}, $y$ is unstable for $Z$. But $Z$ is of
    finite length, so there is $y \in C \setminus Z$ and $Z$ is not Torelli.
  \end{proof}

Recall Dolgachev's conjecture from the introduction (see
\cite[Conjecture 5.8]{dolgachev:logarithmic}).
It states that a semi-stable
arrangement of hyperplanes $Z$ (i.e. such that $\sF_Z$ is a
semi-stable sheaf) fails to be Torelli if and only if $Z$ belongs to a
stable rational curve of degree $n$.

  \begin{corol}
    The ``only if'' implication of Dolgachev's conjecture is true. 
  \end{corol}

  \begin{proof}
    If $Z$ belongs to a curve $C = C_0 \cup \cdots \cup C_s$ as above,
    then we fix one component $C=C_0$ 
    and we define $L_i$ as the span of $C_i$, for $i>0$. The variety $Y = C\cup L_1 \cup \cdots
    \cup L_s$ is a KW variety containing $Z$, so $Z$ is not Torelli.
  \end{proof}

\begin{corol}
  A finite length subscheme $Z$ of $\p^2$, whose set-theoretic support
  is not contained in a
  line, is Torelli if and only if it is not contained in a conic.
\end{corol}
Hence Dolgachev's conjecture (see \cite[Chapter
5]{dolgachev:logarithmic}) holds on 
$\p^2$. In fact something quite stronger is 
true, for no stability condition is required in our result; in
fact $\sF_Z$ needs not even be torsion-free.

We note in the next corollary that,
for generic arrangements, our approach gives a quick proof of
some of the main 
results of \cite{dolgachev-kapranov:arrangements} and \cite{valles:instables}.
Also, we note some simple examples of non-generic Torelli arrangements. 

\begin{corol}
  Let $Z$ be a subscheme of length $\ell < \infty$ of $\p_n$.
  \begin{enumerate}[i)]
  \item \label{T1} If the subscheme $Z$ is contained in no quadric,
    then $Z$ is Torelli; 
  \item \label{T3} assume that $Z$ is in general linear position and
    $\ell \ge n+3$.
    Then $Z$ is contained in a smooth rational normal curve of
    degree $n$ if and only if $Z$ is not Torelli.
  \end{enumerate}
\end{corol}

\begin{proof}
  The statement \eqref{T1} is clear,
  since all $2\times 2$ minors of the matrix $M$ of the
  previous lemma are quadrics.

  For \eqref{T3}, 
  we want to show that, if $\ell \ge n+3$ and $Z$ is in general linear
  position, then $Z$ is contained in a KW variety $Y$ if and only if it is contained in a rational
  normal curve of degree $n$. One direction is clear, so we assume that 
  there are $C$, $L_1,\ldots,L_s$ as in 
  Theorem \ref{main}, such that $Y = C \cup L_1 \cup \cdots \cup L_s$ contains
  $Z$, with $s\geq 1$.
  Note that the span $L'$ of $C \cup L_1 \cup \cdots \cup L_{s-1}$ has
  dimension $d+a_1+\cdots+a_{s-1}$, hence there are at most
  $d+a_1+\cdots+a_{s-1}+1$ points of $Z$ in $L'$. Also, $L_s$
  contains at most $a_s+1$ points of $Z$. Hence $Y$ contains at most 
  $d+a_1+\cdots+a_{s}+2 = n+2$ points of $Z$, so $\ell \ge n+3$
  contradicts that $Z$ be contained in $Y$.
\end{proof}

\subsection{Maximal number of unstable hyperplanes}

One can ask, given a Steiner sheaf $\sE$, how to recognize if 
$\sE$ is isomorphic to $\sF_Z$, for some $Z$ in $\p_n$. 
The next theorem gives an answer to this question.

\begin{main} \label{max}
Let $\sE$ be a sheaf having resolution:
\[
0 \to \sO_{\p^n}(-1)^{\ell-n-1} \to \sO_{\p^n}^{\ell-1} \to \sE \to 0.
\]
Assume that $\W(\sE)$ contains $\ell$ distinct points
$\{z_1,\ldots,z_\ell\}=Z$, and that $\sO_{H_{z_i}}$ is not a direct
summand of $\sE$, for any $j$.
Then $\sE$ is isomorphic to $\sF_Z$.
\end{main}

\begin{proof}
  Let $H$ be an unstable hyperplane of $\sE$, hence assume
  $\HH^{n-1}(H,\sE_{|H}(-n)) \neq 0$, i.e. $\HH^{n-1}(\p^n,\sE \ts \sO_H(-n)) \neq 0$.
  We have:
  \begin{align*}
  \HH^{n-1}(\p^n,\sE \ts \sO_H(-n)) & \cong
  \Ext^{n-1}_{\p^n}(\sO_{\p^n},\sE \ts \RRHHom(\sO_H(n+1)[-1],\sO_{\p^n})) \cong \\
  & \cong \Ext^{n-1}_{\p^n}(\sO_H(n+1)[-1],\sE) \cong \\
  & \cong \Hom_{\p^n}(\sE,\sO_H)^*.
  \end{align*}
  
  Looking at the resolutions of $\sE$ and $\sO_H$, one sees that any
  non-zero map $\sE \to \sO_H$ is surjective, and that 
  the kernel $\sE'$ of such a map is again a Steiner sheaf.

  Let now $H' \neq H$ be another unstable hyperplane of $\sE$. By
  the induced map $\HH^{n-1}(H',\sE'_{|H'}(-n)) \epi
  \HH^{n-1}(H',\sE_{|H'}(-n))$ we see that $H'$ is unstable for $\sE'$
  as well. Let $\sK$ by the kernel of the (surjective) map $\sE' \to
  \sO_{H'}$.
  Then $\sK$ injects in $\sE$, and we let $\sC$ be $\sE/\sK$.
  We claim that $\sC$ is isomorphic to $\sO_H \oplus \sO_{H'}$.
  Indeed, we have $\sE'/\sK \cong \sO_{H'}$, hence we get an exact
  sequence:
  \[
  0 \to \sO_{H'} \to \sC \to \sO_H \to 0.
  \]
  Switching the roles of $H$ and $H'$ provides a splitting of the
  above sequence, so that $\sC \cong \sO_H \oplus \sO_{H'}$.
  
  Iterating this procedure, we find a surjective map:
  \[
  \sE \epi \bigoplus_{i=1,\ldots,\ell} \sO_{H_{z_i}}.
  \]
  Note that the kernel of this map is $\Omega_{\p^n}$.
  Indeed, by diagram chasing, it is the kernel of a  surjective map
  $\sO_{\p^n}(-1)^{n+1} \epi \sO_{\p^n}$.
  Therefore we have and exact sequence:
  \[
  0 \to \Omega_{\p^n} \to \sE \to \bigoplus_{i=1,\ldots,\ell}
  \sO_{H_{z_i}} \to 0.
  \]

  To conclude we can use Claim \ref{theclaim}.
  Indeed, $\sF_Z$ is given, up to isomorphism, as
  the only extension of $\bigoplus_{i=1,\ldots,\ell}
  \sO_{H_{z_i}}$ by $\Omega_{\p^n}$ associated by Claim \ref{theclaim}
  to the canonical surjection $\sO_{\p_n} \to \sO_Z$.
  An extension of $\bigoplus_{i=1,\ldots,\ell}
  \sO_{H_{z_i}}$ by $\Omega_{\p^n}$ not isomorphic to $\sF_Z$
  corresponds then to a map $\sO_{\p_n} \to \sO_Z$ which is not
  surjective, say $\sO_{z_j}$ is not in the image.
  Such extension contains $\sO_{H_{z_j}}$ as a direct summand, which
  contradicts our hypothesis on $\sE$.
\end{proof}

We get the following bound on the number of unstable hyperplanes of a
Steiner sheaf.

\begin{corol}
  Let $\sE$ be a torsion-free Steiner sheaf with resolution:
  \[
  0 \to \sO_{\p^n}(-1)^{\ell-n-1} \to \sO_{\p^n}^{\ell-1} \to \sE \to 0.
  \]
  Assume that $\W(\sE)$ contains $\ell$ distinct points
  $\{z_1,\ldots,z_\ell\}=Z$ not contained in a KW variety in $\p_n$.
  Then $\W(\sE) = Z$.
\end{corol}

The following proposition gives an elementary way to write down the
matrix $M_Z$.

\begin{prop}
  Let $Z = \{z_1,\ldots,z_\ell\}$ be a non-degenerate Torelli arrangement, and
  consider the equations $f_1,\ldots,f_\ell$ of the $\ell$ hyperplanes
  of $\p^n$.
  Then, up to permutation of $1,\ldots,\ell$,
  there are constants $\alpha_{i,j}$ such that:
  \begin{equation}
    \label{easy}
    f_\ell = \sum_{i = 1,\ldots,\ell-1} \alpha_{i,j} f_i,    
  \end{equation}
  for all $j=1,\ldots,\ell-n-1$, and the matrix $M_Z$ can be written as:
  \[
  M = \left( \begin{array}{ccc}
    \alpha_{1,1} f_1 & \cdots & \alpha_{\ell,1} f_{\ell-1} \\
    \vdots && \vdots \\
    \alpha_{1,\ell-n-1} f_1 & \cdots & \alpha_{\ell,\ell-n-1} f_{\ell-1}
    \end{array} \right).
  \]
\end{prop}

\begin{proof}
  The $\ell$ forms $f_1,\ldots,f_\ell$ span the space $V$ that
  has dimension $n+1$, hence up to reordering there
  are $\ell-n-1$ linearly independent ways of writing $f_\ell$ as combination of
  $f_1,\ldots,f_{\ell-1}$, and we have the constants $\alpha_{i,j}$.

  Now, the $i$-th column of the matrix $M$ above vanishes identically on
  the hyperplane $H_i$, which implies that $H_i$ is unstable for the
  cokernel $\sE$ of $M\tra$ for $i=1,\ldots,\ell-1$. Further, in view of \eqref{easy}, we have
  that $H_\ell$ is also unstable for $\sE$.
  Therefore, since $Z$ is Torelli we conclude that $\W(\sE)=Z$, hence,
  by the previous theorem,  $M_Z$ can be taken to be precisely $M$.
\end{proof}

\section{Decomposition of logarithmic sheaves}
\label{section-decomposition}

Here we develop a tool for studying semistability of non-Torelli
arrangements. This tool will take the form of a filtration associated
to any non-Torelli arrangement. We will use this to provide some exceptions
to Dolgachev's conjecture.

\subsection{Blowing up a linear subspace}

Let $U$ be a $k+1$-dimensional subspace of $V$, with $1 \le k \le
n-1$, and consider the subspace $\p_k=\p(U^*)$ of $\p_n=\p(V^*)$, embedded
by $i: \p(U^*) \mono \p_n$. Define $U^\perp$ as the kernel of 
the projection $V^* \to U^*$, and note that $U^\perp \cong (V/U)^*$.
Denote by $\tilde{\p}_U^n$ the blowing up of $\p^n$ along
$\p^{n-k-1}=\p(V/U) \subset \p^n$, and write 
$\pi_U : \tilde{\p}^n \to \p^k$ and $\sigma_U : \tilde{\p}^n \to \p^n$ for the two
natural projections (we will drop this index $U$ whenever possible).
In our convention, points of $\p(V)$ and $\p(U)$ are quotients of $V$
and $U$, so one can write:
\[
\tilde{\p}^n = \{(x,u) \in \p^n \times \p^k \, | \, x_{|U} = u\}.
\]

We consider $\F^k_k = \{(u,v) \in \p^k\times \p_k \, | \, u \in H_v\}$
and  $p_U$ and $q_U$ are the natural projections to $\p^k$ and $\p_k$.
In order to compare the incidence varieties $\F^n_n$ over $\p^n$
and $\F^k_k$ over $\p^k$, we consider the blown-up flag:
\[
\tilde{\F}^n_n = \{(x,u,y) \in \p^n \times \p^k \times \p_n \, 
| \, x_{|U} = u, x \in H_y\}.
\]
This blown-up flag contains the relative blown-up flag:
\[
\tilde{\F}^n_k = \{(x,u,v) \in \p^n \times \p^k \times \p_k \, 
| \, x_{|U} = u, x \in H_v\}.
\]
Projecting onto the different coordinates we get the commutative
diagrams:
\begin{equation}
  \label{diagrammoni}
\xymatrix{
& \tilde{\F}^n_n \ar[r] \ar[d] & \F^n_n \ar^{p}[d] &&  & \F^k_k \ar^{p_U}[dr] \ar[d] \\
\F^k_k  \ar[r]  \ar_{q_U}[dr]  \ar@{^(->}[ur] & \tilde{\p}^n \ar^\sigma[r] \ar^{\pi}[d] & \p^n && \tilde{\F}^n_k  \ar[ur] \ar@{^(->}[d] \ar[r] & \tilde{\F}^n_k \ar@{^(->}[d]  \ar[r] & \p_k \ar^{i}@{^(->}[d] \\
& \p^k &&& \tilde{\F}^n_n  \ar[r] & \F^n_n  \ar_{q}[r] & \p_n
}  
\end{equation}

Let us analyze the sheaf $\sF_Z$ when $Z$ is
degenerate, namely $Z$ spans a proper subspace $\p(U^*) = \p_k \subset
\p_n$. We may think that the last $n-k$ coordinates in $\p_n$ vanish
on $\p_k$.
This amounts to ask that the equations of the hyperplanes of $Z$ only
depend on the variables $x_0,\ldots,x_k$.
The same happens to the matrix $M_Z$, that now naturally defines 
the Steiner sheaf $\sF^U_Z$ over $\p^k$
associated to $Z \subset \p_k$.
Note that we have the rational map:
\[
\rho : \p^n \dashrightarrow \p^k
\]
It is tempting to look at $\rho^*(\sF_Z^U)$ as a component of $\sF_Z$,
defined by the same matrix $M_Z$, pulled back on $\p^n$ by $\rho$.
The following lemma proves that this can be done (up to resolving the
indeterminacy of $\rho$), and that the
remaining component is $(n-k)$ copies of $\sO_{\p^n}(-1)$.

\begin{lem} \label{degenerate}
  Let $Z$ be a finite length subscheme of $\p_n$, assume that $Z$
  spans a $\p_k = \p(U^*)$ with $1 \le k \le
  n-1$, and let $\sigma  = \sigma_U, \pi = \pi_U$.
  Then we have:
  \[
  \sF_Z \cong V/U \ts \sO_{\p^n}(-1) \oplus  \sigma_* \pi^*(\sF^U_Z).
  \]
\end{lem}

\begin{proof}
  Assume that $Z$ is contained in $\p_k = \p(U^*)$ and consider the exact
  sequence:
  \[
  0 \to \cI_{\p_k,\p_n}(1) \to \cI_{Z,\p_n}(1) \to i_*(\cI_{Z,\p_k}(1)) \to 0,
  \]
  and the Koszul complex resolving $\cI_{\p_k,\p_n}(1)$, namely:
  \[
  0 \to \sO_{\p_n}(k-n+1) \to \cdots \to
  \wedge^2 U^\perp \ts \sO_{\p_n}(-1) 
  \to U^\perp \ts \sO_{\p_n} \to \cI_{\p_k,\p_n}(1) \to 0.
  \]
  Applying $\RR p_*(q^*(-))$ to these exact sequences,
  in view of the vanishing $\RR p_*(q^*(\sO_{\p_n}(t)))$ for $2-n \le
  t \le -1$, we get a distinguished triangle:
  \[
  U^\perp \ts \sO_{\p^n} \to \RR p_*q^*(\cI_{Z}(1)) \to \RR p_*q^*(i_*(\cI_{Z,\p_k}(1)))
  \shift
  \]
  Taking $\RRHHom_{\p_n}(-,\sO_{\p_n}(-1))$, we obtain the distinguished triangle:
  \[
  \RRHHom_{\p_n}(\RR p_*q^*(i_*(\cI_{Z,\p_k}(1))),\sO_{\p_n}(-1)) \to \sF_Z \to V/U \ts \sO_{\p^n}(-1)
  \shift
  \]

  Our task is thus to prove that the leftmost complex in the triangle
  above is a sheaf isomorphic to $\sigma_* \pi^*(\sF^U_Z)$. 
  Let $\sE_Z$ be this complex, for the remaining part of the proof.

  Using repeatedly commutativity of the 
  diagrams \eqref{diagrammoni} together with projection formula, it is easy to get a natural transformation:
  \[
  \RR \sigma_* (\RR \tilde{p}_U)_*\alpha^*q_U^* \cong  \RR p_*q^*i_*,
  \]
  where $\alpha$ is the projection $\tilde{\F}^n_k \to \F^k_k$.
  By smooth base change, we also have:
  \[
  (\RR \tilde{p}_U)_*\alpha^* \cong \pi^* (\RR {p}_U)_*,
  \]
  where $\tilde{p}_U$ is the projection $\tilde{\F}^n_n \to \tilde{\p}^n$.
  This gives at once the 
  natural isomorphism: 
  \begin{equation}
    \label{projection}
    \RR \sigma_* \pi^* (\RR p_U)_*q_U^*(\cI_{Z,\p_k}(1))  \cong \RR
    p_*q^*i_*(\cI_{Z,\p_k}(1)).
  \end{equation}
  Therefore, in order to compute $\sE_Z$, we have to apply
  $\RRHHom_{\p^n}(-,\sO_{\p^n}(-1))$ to the left hand side.
  But we have seen that this simply amounts to transpose a matrix
  of linear forms of size $(\ell-1) \times (\ell - k - 1)$,
  just as well as transposition is needed to 
  define $\sF_Z^U$ from $\RR (p_U)_*q_U^*(\cI_{Z,\p_k}(1))$ on $\p^k$,
  so that dualization of these complexes commutes with taking $\RR \sigma_* \pi^*$.
  Hence we have shown that $\sE_Z$ is isomorphic to $\RR \sigma_*
  \pi^*(\sF^U_Z)$, and therefore to $\sigma_*
  \pi^*(\sF^U_Z)$.

  This provides a short exact sequence:
  \[
  0 \to \sigma_* \pi^*(\sF^U_Z) \to \sF_Z \to  V/U \ts
  \sO_{\p^n}(-1)\to 0.
  \]
  We will be done once this sequence splits, which in turn would be ensured
  by the vanishing:
  \[
  \Ext^1_{\p^n}(\sO_{\p^n}(-1),\sigma_* \pi^*(\sF^U_Z)) = 0.
  \]
  But this vanishing is clear since $\sigma_* \pi^*(\sF^U_Z)$ is a
  Steiner sheaf.
 \end{proof}
In the above situation, we set:
\[ \sE_Z^U = \sigma_* \pi^*(\sF^{U}_{Z}). \]

\subsection{Decomposing non-Torelli arrangements}

Let us borrow the notations from the previous paragraph.
In particular, recall that, given a $(k+1)$-dimensional subspace $U$
of $V$, and $Z$ in $\p(U^*)$, we have a sheaf $\sF_Z^U$ over $\p(U)$,
and hence a sheaf $\sigma_* \pi^*(\sF_Z^U)$ over $\p^n=\p(V)$,
where $\sigma=\sigma^U$ and $\pi=\pi_U$ are the natural projections to $\p^n$ and
$\p(U)$ from the blow-up $\tilde{\p}^n$ of $\p^n$ along $\p(V/U)$.

\begin{lem} \label{curve}
Assume that $Z$ is contained in a rational normal curve $C$ spanning
$\p(U^*) \subset \p_n$. Then $\sF_Z^U$ is isomorphic to $\sF_{Z'}^U$, for
any other subscheme $Z'$ contained in $C$ having the same length as
$Z$.  
\end{lem}

\begin{proof}
  Let $\ell$ be the length of $Z$.
  We consider the  exact sequence:
  \[
  0 \to \cI_{C,\p(U^*)}(1) \to \cI_{Z,\p(U^*)}(1) \to \sO_{C}((d-\ell)p) \to 0,
  \]
  where, given an integer $a$, we write $\sO_{C}(a p)$ for a divisor
  of degree $a$ in $C$, 
  namely $a$ times a point $p \in C \cong \p^1$.
  Looking at the sheafified minimal graded free resolution of
  $\cI_{C,\p(U^*)}(1)$ over $\p(U^*)$, we see immediately that:
  \[
  \RR (p_U)_*q_U^*(\cI_{C,\p(U^*)}(1)) = 0.
  \]
  
  Therefore the complex $  \RR (p_U)_*q_U^*(\cI_{Z,\p(U^*)}(1))$ only
  depends on the value $\ell$, hence so does $\sF_Z^U$.
\end{proof}

By the previous lemma, if $C_d$ is 
a rational normal curve of degree $d$ spanning a $\p_d=\p(U^*)$, we can set:
\[
\sE^{C_d}_{\ell} = \sigma_* (\pi^* (\sF_{Z}^U)), \qquad
\]
for any subscheme $Z$ of length $\ell$ of $C_d$.

The next result gives a decomposition tool for an 
arrangement $Z$ which is contained in a KW-variety $Y$.
So, let $Y = C \cup L_1 \cup \cdots \cup L_s$, 
where $L_i=\p(U_i)=\p_{n_i}$ and $C$ is a smooth rational curve of
degree $d>0$, and the conditions 
\eqref{I} and \eqref{II} of the introduction are satisfied.
Let $y_i = C \cap L_i$.

\begin{main} \label{thm:decomposition}
  Let $Z = Z_0 \cup \cdots \cup Z_s \subset \p_n$ be a subscheme of
  length $\ell$, smooth at $y_i$ for all $i$.
  Assume that $L_i$ is the span of $Z_i$, and 
  that $Z_0 \subset C \setminus \{y_1,\ldots,y_s\}$.
  Set $\ell_i$ for the length of $Z_i$. Then:
  \begin{enumerate}[i)]
  \item \label{primo} we have a natural exact sequence:
  \begin{equation}
    \label{D}
  0 \to \bigoplus_{i=1,\ldots,s} \sE^{U_i}_{Z_i} \to \sF_Z  
  \to \sE^{C_d}_{\ell_0+s} \to 0;
  \end{equation}
  \item \label{secondo} we have the resolutions:
  \begin{align*}
    & 0 \to \sO_{\p^n}(-1)^{\ell_i - n_i-1} \to \sO_{\p^n}^{\ell_i-1}
    \to \sE^{U_i}_{Z_i} \to 0, \\
    & 0 \to \sO_{\p^n}(-1)^{\ell_0+s-d-1} \to \sO_{\p^n}^{\ell_0+s-1}
    \to \sE^{C_d}_{\ell_0+s} \to 0.
  \end{align*}
  \end{enumerate}
\end{main}

\begin{proof}
  Since $Z$ lies in $Y = C \cup L_1 \cup \cdots \cup L_s$, we have the
  sequences:
  \begin{align} \label{ideali}
    & 0 \to \cI_{Y,\p_n}(1) \to \cI_{Z,\p_n}(1) \to \cI_{Z,Y}(1) \to 0.
  \end{align}

  The following claim ensures that $\cI_{Y,\p_n}(1)$ does not
  contribute to $\sF_Z$.

  \begin{claim} \label{Y}
    Given $Y$ as above, we have $\RR p_* q^*(\cI_{Y,\p_n}(1))=0$.
  \end{claim}

  Let us postpone the proof of the claim above, and assume it for
  the time being.
  Set $\LL = L_1 \cup \cdots \cup L_s$, $Z' = Z_1 \cup
  \cdots \cup Z_s$ and $Z_0' = Z_0 \cup y_1 \cup \cdots \cup y_s$.

  By the definition of $Y$ and the hypothesis on $Z$ we deduce the
  following exact commutative exact diagram:
  \begin{equation}
    \label{diagrammozzo}
  \xymatrix@-2ex{
    & 0 \ar[d] & 0 \ar[d] & 0 \ar[d] \\
    0 \ar[r] & \cI_{Z_0',C}(1) \ar[d] \ar[r] & \sO_C((d-s)p) \ar[r]
    \ar[d] & \sO_{Z_0} \ar[r] \ar[d] & 0\\
    0 \ar[r] & \cI_{Z,Y}(1)  \ar[d] \ar[r] & \sO_Y(1) \ar[r] \ar[d] &
    \sO_{Z} \ar[r]  \ar[d] &  0\\
    0 \ar[r] & \cI_{Z',\LL}(1) \ar[d] \ar[r] & \sO_{\LL}(1) \ar[r] \ar[d] & \sO_{Z'}
    \ar[r] \ar[d] & 0\\
    & 0 & 0 & 0
  }
  \end{equation}
  Here, $p$ is a point in $C \cong \p^1$.
  Moreover, clearly we have:
  \begin{equation}
    \label{ideali-4}
  \cI_{Z',\LL}(1) \cong \bigoplus_{i=1,\ldots,s} \cI_{Z_i,L_i}(1).
  \end{equation}
  Hence, we may rewrite the leftmost column of the above diagram as:
  \begin{equation}
    \label{ideali-3}
     0 \to \sO_C((-s-\ell_0+d) p) \to \cI_{Z,Y}(1) \to
    \bigoplus_{i=1,\ldots,s} \cI_{Z_i,L_i}(1) \to  0.
  \end{equation}

  Notice also that we can switch the roles of $C$ and $\LL$, to obtain:
  \begin{align}
    \label{ideali-2} & 0 \to \bigoplus_{i=1,\ldots,s} \cI_{y_i,L_i}(1) \to \sO_{Y}(1)
    \to \sO_{C}(1) \to 0.
  \end{align}

  Applying the functor $\RR p_*(q^*(-))$ to the exact sequence
  \eqref{ideali} and 
  dualizing, we have, 
  in view of Claim \ref{Y}:
  \[
  \sF_Z \cong  \RRHHom_{\p^n}(\RR p_*(q^*(\cI_{Z,Y}(1))),\sO_{\p^n}(-1)).
  \]

  Applying $\RR p_*(q^*(-))$
  and $\RRHHom_{\p^n}(-,\sO_{\p^n}(-1))$ to \eqref{ideali-3}
  gives the desired exact sequence \eqref{D}. Indeed,
  For each of the terms $\cI_{y_i,L_i}(1)$ appearing in
  the isomorphisms \eqref{ideali-4}, we can use the argument
  used in Lemma \ref{degenerate}, that gives:
  \[
  \RRHHom_{\p^n}(\RR p_*(q^*(\cI_{y_i,L_i}(1))),\sO_{\p^n}(-1)) \cong 
  \sigma^{U_i}_* \pi_{U_i}^*(\sF^{U_i}_{Z_i}) = 
  \sE_{Z_i}^{U_i}.
  \]
  For $\sO_C(d - \ell_0 - s)$ we use 
  the same argument and Lemma \ref{curve} to obtain:
  \[
  \RRHHom_{\p^n}(\RR p_*(q^*(\sO_C(d - \ell_0 - s))),\sO_{\p^n}(-1)) \cong 
  \sigma^{U_0}_* \pi_{U_0}^*(\sF^{U_0}_{Z_0'}) = 
  \sE_{\ell_0 + s}^{C_d}.
  \]

  We thus proved \eqref{primo}.
  The resolutions required for \eqref{secondo}
  are provided by Lemma \ref{degenerate}.
  It remains to prove Claim \ref{Y}.
\end{proof}

\begin{proof}[Proof of Claim \ref{Y}]
  Looking at \eqref{incidence}, we see that the claim follows if we
  prove that $\cI_Y(1)$ is the cohomology of a complex where only the
  sheaves $\sO_{\p_n}(1-n), \ldots , \sO_{\p_n}(-1)$ appear.
  We can use Beilinson's theorem to prove that this is the case.
  In fact we merely have to prove the following vanishing results:
  \begin{equation}
    \label{vanishing}
    \HH^k(\p_n,\cI_Y(t)) = 0, \qquad \mbox{for all $k$, and for $t = 0,1$}.
  \end{equation}
  
  To show this, we look at \eqref{ideali-2}.
  Since $d + n_1 + \cdots + n_s = n$,
  taking cohomology of this sequence, we get :
  \[  
  \HH^k(\p_n,\sO_Y(1)) = 0, \qquad \mbox{for all $k > 0$},     
    \qquad \dim_\kk \HH^0(\p_n,\sO_Y(1)) = n+1.  
    \]
  Hence we have \eqref{vanishing} for $t=1$, for $Y$ is non-degenerate.

  Taking cohomology of \eqref{ideali-2}, twisted by $\sO_{\p_n}(-1)$
  immediately gives \eqref{vanishing} for $t=0$, and we are done.
\end{proof}

\begin{corol} \label{split}
  With the notations of the previous theorem, 
   $\sE_{Z_i}^{U_i}$ is a direct summand of $\sF_Z$ if
  $y_i$ belongs to $Z$.
\end{corol}

\begin{proof}
  Order $1,\ldots,s$ so
  that $y_1,\ldots,y_r$ belong to $Z$ and $y_{r+1},\ldots,y_s$ do
  not.
  Using \eqref{ideali-2} and a diagram similar to \eqref{diagrammozzo},
  we get an exact sequence:
  \[
  0 \to \bigoplus_{i=1,\ldots,r} \cI_{Z_i, L_i}(1) \oplus
  \bigoplus_{i=r+1,\ldots,s} \cI_{Z_i \cup y_i, L_i} (1)
 \to \cI_{Z,Y}(1) \to  
  \sO_C((d-r-\ell_0)p) \to 0.
  \]

  Comparing with \eqref{ideali-3}, we see that, for $i=1,\ldots,r$,
  $\cI_{Z_i, L_i}(1)$ is a direct summand of $\cI_{Z,Y}(1)$, so that 
  $\sE_{Z_i}^{U_i}$ is a direct summand of $\sF_Z$.
\end{proof}


\subsection{Exceptions to Dolgachev's conjecture}

We conclude the paper with some examples of 
hyperplane arrangements having interesting unstable loci,
giving some counterexamples to 
the ``only if'' implication of Dolgachev's conjecture.
Namely, we describe
finite sets $Z$ in $\p_n$ such that $\W(\sF_Z)$ is the union of $Z$
and a line in $\p_3$, or $Z$ and a plane in $\p_4$, or $Z$ and a point in $\p_4$.
The results of this section are used to prove semistability in some cases.

\begin{eg} \label{NO}
  We consider the union $Z_1$ of $5$
  points on a unique conic, spanning 
  a plane $L_1$ in $\p_3$, and the union $Z_0$ of $2$ more
  points on a line $L_0$.
  We assume that  $L_0$ does not meet
  the conic $D\subset L_1$ passing through $Z_1$, and that $Z_0 \cap
  L_1 = \emptyset$. We let $Z = Z_0 \cup Z_1$.
  \begin{figure}[h!]
    \centering
    \includegraphics[height=1.5in]{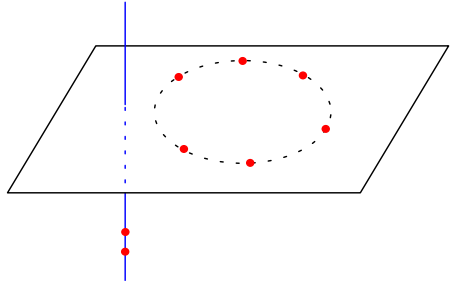}
    \caption{Seven points in $\p_3$ with an unstable line.}
    \label{fig:7-points}
  \end{figure}

 Consider a point $y$ of $L_0$. Then there are a rational normal curve
 through $y$ (take $L_0$) and a plane (take $L_1$) such that $L_0 \cup
 L_1$ contains $Z$, and satisfying \eqref{I} and \eqref{II}.
 Thus all points of $L_0$ are unstable, and $Z$ is not Torelli.

 On the other hand, if $y \not \in Z$ does not lie in $L_0$, then
 $y$ is not unstable for $\sF_Z$. Indeed, any subvariety $Y \subset
 \p_n$ through $y$ and $Z$ as in Theorem \ref{main} would have to
 contain $Z_1$ and $L$, hence be $L_0 \cup L_1$.
 So $y$ has to lie in
 $L_1$.
 But even the points of
 $L_1 \setminus Z$ are not unstable, for we should have a conic in
 $L_1$ through $y$ and $Z_1$ (hence the conic is $D$) and a line
 through $Z_1$ (hence the line is $L_0$) meeting at a single point; 
 but $D$ does not pass through $L_0 \cap L_1$.

 Finally, note that $\sF_Z$ is a stable sheaf, at least for most
 choices of the $5$ points of $Z_1$.
 In fact, let us prove it under the assumption
 that $Z_1=\{\zeta_1,\ldots,\zeta_5\}$ is such that $\zeta_3$ lies in intersection
 of the lines $N_1$ and $N_2$ through $\zeta_1,\zeta_2$ and $\zeta_4,\zeta_5$ (still 
 $D = N_1 \cup N_2$ disjoint from $L_0$).
 In this case, Theorem \ref{thm:decomposition} applies to
 give a short exact sequence:
 \[
 0 \to \sF_1 \to \sF_Z \to \sF_0 \to 0,
 \]
 where $\sF_1$ is $\sE^{U_1}_{Z_1}$ (we set $L_i = \p(U_i)$) and
 $\sF_0$ is $\sE^{L_0}_{-3}$, which in this case is isomorphic to
 $\cI_{M_0}(1)$, where $M_0$ is the line dual to $L_0$.
 Here 
 $\sF_1$ splits, in view of Corollary \ref{split}, as $\cI_{M_1}(1)
 \oplus \cI_{M_2}(1)$, where the $M_i$'s are the lines dual to the $N_i$'s.
 Then, it is straightforward to check that $\sF_Z$ is
 strictly semistable, for the graded object associated to the above
 filtration of $\sF_Z$ is $\cI_{M_0}(1) \oplus \cI_{M_1}(1) \oplus
 \cI_{M_2}(1)$.

 In coordinates, we could take $L_0$ as $\{z_2 = z_3 = 0\}$ and $L_1$
 as $\{z_1 = 0\}$. Further, $N_1$ and $N_2$ can be taken as
 $\{z_0-z_2 = z_1 = 0\}$ and $\{z_0-z_3 = z_1 = 0\}$, so that $\zeta_3
 = (1:0:1:1)$.
 The matrix $M_Z$ in this case is:
   \[
  M_Z=\left (
 \begin{array}{cccccc}
   x_0+x_1 & -x_1 & 0 & x_3 & 0 & x_2  \\
   0 & 0 & x_0+x_2 & x_3 & 0 & 0\\
   0 & 0 & 0 & 0  & x_0+x_3 & x_2 
 \end{array} \right ),
 \]
\end{eg}

\begin{eg} \label{NO2}
  With a little more work one can modify the above example so that
  $\sF_Z$ is even stable.
  This can be achieved adding a point on $L_0$ and a further point on
  $L_1$, outside $N_1 \cup N_2$.

  In coordinates, we can add $(1:2:0:0)$ and $(0:0:1:1)$. This gives
  rise (up to permutation) to the matrix $M_Z$:
   $$\bgroup\begin{pmatrix}{x}_{0}+{x}_{1}&
      0&
      {-{x}_{1}}&
      0&
      {x}_{3}&
      0&
      {x}_{2}&
      0\\
      0&
      {x}_{0}+2 {x}_{1}&
      {-2 {x}_{1}}&
      0&
      {x}_{3}&
      0&
      {x}_{2}&
      0\\
      0&
      0&
      0&
      {x}_{0}+{x}_{2}&
      {x}_{3}&
      0&
      0&
      0\\
      0&
      0&
      0&
      0&
      0&
      {x}_{0}+{x}_{3}&
      {x}_{2}&
      0\\
      {x}_{0}+{x}_{1}&
      0&
      {-{x}_{1}}&
      0&
      0&
      0&
      0&
      {x}_{2}+{x}_{3}\\
      \end{pmatrix}\egroup$$
      
      Stability of $\sF_Z$ can be deduced by the following resolutions:
      \begin{align*}
        & 0 \to \sO_{\p^3}(-3) \oplus \sO_{\p^3}(-2) \to \sO_{\p^3}(-2) \oplus \sO_{\p^3}(-1)^4 \to \sF_Z^{**}(-2) \to 0, \\
        & 0 \to \sO_{\p^3}(-4) \to \sO_{\p^3}(-3) \oplus \sO_{\p^3}(-1)^3 \to \sF_Z^*(1) \to 0.
      \end{align*}
\end{eg}

\begin{eg}
  Let $L_1$ and $L_2$ be two planes in $\p_4$, meeting at a single
  point $y$. Then $y$ is the distinguished point of the KW
  variety $L_1 \cup L_2$. Let $Z_1 \subset L_1$ and $Z_2 \subset L_2$
  be subschemes of length $\ell_1,\ell_2 < \infty$, both disjoint from $y$.
  Then $Z = Z_1 \cup Z_2$  cannot be Torelli, for $y$ is always an
  unstable hyperplane of $\sF_Z$.

  If there is no conic through $Z_1$ and $y$ nor through $Z_2$ and
  $y$, then $y$ is the {\it only point of $\p_4$ outside $Z$ giving an
    unstable hyperplane for $\sF_Z$}.
  If $Z_1$ consists of $3$ points such that $Z_1 \cup y$ is in
  general linear position, then for a general point $z$ of $L_1$,
  there is a conic $C$ through $z \cup y \cup Z_1$, and $Z$ is
  contained in the KW variety $C \cup L_2$. Hence any point
  of $C$ is unstable. So {\it all the points of $L_1$ give unstable
    hyperplanes} in this case.
\end{eg}


\def\cprime{$'$} \def\cprime{$'$} \def\cprime{$'$} \def\cprime{$'$}
  \def\cprime{$'$} \def\cprime{$'$} \def\cprime{$'$}
\providecommand{\bysame}{\leavevmode\hbox to3em{\hrulefill}\thinspace}
\providecommand{\MR}{\relax\ifhmode\unskip\space\fi MR }
\providecommand{\MRhref}[2]{
}
\providecommand{\href}[2]{#2}

\end{document}